\documentclass[twoside]{article}
\usepackage{graphicx}
\usepackage{amsmath,amssymb,amsthm}
\usepackage{cite,hyperref}
\usepackage{geometry}[margins=1in]

\usepackage{tikz,pgfplots}
        \pgfplotsset{compat = 1.3}
        \pgfplotsset{minor grid style={dotted}}
        \pgfplotsset{major grid style={dashed}}
        \pgfplotsset{every x tick label/.append style={font=\footnotesize, yshift=0.25ex}}
        \pgfplotsset{every y tick label/.append style={font=\footnotesize, xshift=0.25ex}}

\newtheorem{theorem}{Theorem}
\newtheorem{remark}{Remark}
\newtheorem{prop}{Proposition}
\newtheorem{Lemma}{Lemma}
\newtheorem{corollary}{Corollary}
\counterwithin{table}{section}
\numberwithin{equation}{section}
\usepackage{caption}
\usepackage{fancyhdr}
\newcommand{\uu}{\mathbf{u}}

\begin{document}

\title{Numerical conservation laws of time fractional diffusion PDEs}

\author{A. Cardone, G. Frasca--Caccia}
\maketitle

\abstract{The first part of this paper introduces sufficient conditions to determine conservation laws of diffusion equations of arbitrary fractional order in time. Numerical methods that satisfy a discrete analogue of these conditions have conservation laws that approximate the continuous ones. In the second part of the paper, we propose a method that combines a finite difference method in space with a spectral integrator in time. The time integrator has already been applied in literature to solve time fractional equations with Caputo fractional derivative of order $\alpha\in(0,1)$. It is here generalised to approximate Caputo and Riemann-Liouville fractional derivatives of arbitrary order. We apply the method to subdiffusion and superdiffusion equations with Riemann-Liouville fractional derivative and derive its conservation laws. Finally, we present a range of numerical experiments to show the convergence of the method and its conservation properties.\\}



\section{Introduction}

Conservation laws for differential equations play an important role both in the comprehension of the problem and in the analysis of the mathematical model. A conservation law relates the variation of a certain quantity within an arbitrarily small section of the space domain, to the amount of quantity that flows in and out. This quantity often has a physical meaning, such as mass, energy, momentum, electric charge.
If the system is isolated, the total amount of the conserved quantity does not vary in the evolution of the process. Conservation laws further represent a fundamental tool to study the existence, uniqueness and stability of analytical solutions.

Several mathematical techniques have been developed to construct conservation laws of classical partial differential equations (PDEs) involving integer order derivatives only, including methods based on Noether’s theorem, the direct method, the homotopy operator method, Ibragimov’s method \cite{AncoI,AncoII,Ibra,Olver}. In addition, a great effort has been made for the numerical preservation of conservation laws \cite{cdp16,frasca19,frasca21,Mc14,Wan16,Frasca-Caccia2021,CONTE2022} and to find structure preserving methods \cite{bras21,Conte2,Dahlby,Abdi17,pat14,pat18,pat21,hairer06,mc06,san94,Limbook,vandaele11}.

In the last decade, conservation laws for fractional differential problems have been derived by suitably extending some of these known methods for PDEs. In particular, techniques that rely on generalizations of Noether’s theorem and variational Lie point symmetries have been applied to find conservation laws of fractional differential equations (FDE) with a fractional Lagrangian \cite{cw21,lhrh19}. For nonlinearly self-adjoint FDEs that do not have a Lagrangian in the classical sense, a \emph{formal} Lagrangian can be introduced and conservation laws are obtained by using modern techniques based on Lie group analysis of FDEs. This approach, proposed for the first time by Lukashchuk in 2015 \cite{Luk}, has been applied to time fractional PDEs \cite{Jafari20,lhd18,ahz21,cw21,habibi19,lhrh19,Luk} and more recently to time and space fractional PDEs (see \cite{Gupta17} and references therein).

However, to the best of our knowledge, the specialized literature still misses a study on the numerical preservation of conservation laws of fractional differential problems.

In the present paper, we consider a diffusion equation of fractional order in time, $\alpha$, and give sufficient conditions for identifying its conservation laws. Conservation laws of this equation with $0<\alpha<1$ and $1<\alpha<2$ have been obtained in \cite{Luk}. By exploiting the new result in this paper, we obtain a set of conservation laws for any value of $\alpha$.

The main original result in this paper regards the numerical preservation of these conservation laws. In particular, we show that if a numerical method satisfies a discrete analogue of the sufficient conditions introduced in the continuous setting, then it has discrete conservation laws that approximate the continuous ones. In the integer case, $\alpha=1$, finite difference methods that preserve conservation laws have been introduced in \cite{IMA}.

We propose a mixed method that combines a finite difference scheme along space with a spectral time integrator. With respect to other methods for FDEs known in literature (see e.g. \cite{ccp18,DeLuca20,DeLuca20b,milici2018,mohammadi19,moradi19,moradi20,moradi21,Pod99,vandaele19}), spectral methods present some advantages. In fact, most of these are step-by-step methods and thus they require at each time step a discretization of the long tail of the solution, arising from the hereditary nature of the fractional differential model. Thus, they are computationally expensive. Instead, spectral methods reflect the nonlocal nature of the fractional model and do not involve the discretization of the past history of the solution. Moreover, for a suitable choice of the function basis, spectral methods are exponentially convergent \cite{zk14}.

Here we consider the spectral time integrator proposed in \cite{bcdp17} for time fractional PDEs of order $\alpha\in(0,1)$. This method is here suitably extended to be applied to FDEs of arbitrary fractional order, $\alpha\in(p-1,p)$, with $p\in\mathbb{N}$. 
Discrete conservation laws satisfied by the solutions of this method are derived, and some test examples are presented to highlight the conservation and convergence property of the new scheme.

The rest of the paper is organized as follows. In Section~\ref{sec:materials}, we give some basic material and definitions of fractional differential calculus that are used in the rest of the paper. In Section~\ref{sec:contdisc}, we introduce sufficient conditions to have conservation laws in the continuous and discrete settings. Section~\ref{methsec} generalises the spectral time integrator in \cite{bcdp17} to approximate Riemann-Liouville or Caputo derivatives of arbitrary order. Considering a Riemann-Liouville derivative, in Section \ref{secmodel}, we apply this method to a finite difference discretization in space and derive its conservation laws. In Section \ref{sec:tests}, we verify on some numerical test examples the accuracy of the numerical method and its conservation laws. Finally, some conclusive remarks are drawn in Section \ref{sec:concl}.

\section{Problem setting}\label{sec:materials}

Let us consider a time fractional diffusion PDE of the form
\begin{align}\label{FPDE}
D_t^\alpha u-D_x^q K([u]_x)&=0,
\end{align}
where the symbol $[u]_x$ denotes the function $u$ and its integer derivatives in space and $$p-1<\alpha<p,\qquad p,q\in\mathbb{N},\qquad  u=u(x,t), \qquad (x,t)\in(a,b)\times(t_0,T).$$
We assume that equation (\ref{FPDE}) is complemented by suitable Dirichlet boundary conditions,
\begin{equation}\label{Dir}
u(a,t)=\chi_a(t),\qquad u(b,t)=\chi_b(t),
\end{equation}
and $p$ initial conditions.
Depending on the context, the symbol $D_t^\alpha$ denotes either the Riemann-Liouville fractional derivative,
\begin{equation}\label{RLdef}
^{\text{RL}}D_t^\alpha f=D_t^p(I_t^{p-\alpha}f),
\end{equation}
or the Caputo fractional derivative,
\begin{equation}\label{Cdef}
^{\text{C}}D_t^\alpha f= \left.I_t^{p-\alpha}(D_t^p(f)),\right.
\end{equation}
where
$$I_t^{p-\alpha}f=\frac{1}{\Gamma (p-\alpha)}\int_{t_0}^t\frac{f(\tau,x)}{(t-\tau)^{1-p+\alpha}}\,\mathrm{d}\tau$$
is the Riemann-Liouville integral and $\Gamma(z)$ is the Gamma function. These two definitions of fractional derivative are related by
\begin{equation}\label{relRLC}
^{\text{C}}D_t^\alpha f=  \phantom{.}^{\text{RL}}D_t^\alpha f - \phantom{.}^{\text{RL}}D_t^\alpha \left(\sum_{k=0}^{p-1}\frac{(t-t_0)^k}{k!}f^{(k)}(t_0)\right).
\end{equation}
Thus, if the initial configuration is of total rest the two definitions are equivalent. For more details on the theory of fractional derivatives, we refer the reader to \cite{Pod99}.

If the fractional derivative in (\ref{FPDE}) satisfies the Riemann-Liouville definition (\ref{RLdef}), the initial conditions assigned to (\ref{FPDE}) are \cite{Podlubnypaper}
\begin{equation}\label{IC}
D_t^{\alpha-k} u(x,t_0)=\gamma_k(x),\qquad k=0,\ldots,p-1.
\end{equation}
Instead, if the fractional derivative is of Caputo type (\ref{Cdef}), the initial conditions specify the initial values of the integer derivatives
\begin{equation}\label{ICC}
D_t^{k} u(x,t_0)=\gamma_k(x),\qquad k=0,\ldots,p-1.
\end{equation}

Conservation laws for time fractional diffusion problems have been object of several papers, e.g. \cite{cw21,habibi19,Luk,lhd18,lhrh19}. However, only papers \cite{Luk,lhd18} treat equation of type \eqref{FPDE} defined on a 3D and a 1D space, respectively.

\section{Continuous and discrete con\-ser\-va\-tion laws}\label{sec:contdisc}

\subsection*{Continuous setting}
A conservation law of \eqref{FPDE} is a total divergence,
\begin{equation}\label{CLaw}
D_x(F(x,t,[u]_\alpha))+D_t(G(x,t,[u]_\alpha))
\end{equation}
that vanishes on solutions of \eqref{FPDE}. Functions $F$ and $G$ are called the flux and the density of the conservation law (\ref{CLaw}), respectively. The symbol $[u]_\alpha$ denotes the function $u$, its fractional and integer derivatives and its fractional integrals. Differently from \cite{Luk}, in this paper we assume that $G$ depends on fractional integrals of order $p-\alpha$ only. In fact, integrals of higher order should be treated as new integral variables. Moreover, this is consistent with the limit case of $\alpha$ integer, where $F$ and $G$ are assumed to depend on $u$ and its partial derivatives but not on its integrals \cite{Olver}.

When the boundary conditions are conservative (e.g., periodic) integration in space of (\ref{CLaw}) yields,
$$D_t\int_a^b G(x,t,[u]_\alpha)\,\mathrm{d}x=0,$$
therefore, $$\int_a^b G(x,t,[u]_\alpha)\,\mathrm{d}x$$ is a global invariant of equation \eqref{FPDE}. However, the local conservation law (\ref{CLaw}) holds true regardless of the specific boundary conditions.

The following theorem gives sufficient conditions to identify conservation laws of equation (\ref{FPDE}).

\begin{theorem}\label{Theocont}
If $\rho(t)$ and $\bar{G}=\bar G(x,t,[u]_\alpha)$ are two functions such that
\begin{equation}\label{tclaw}
\rho(t)D_t^\alpha u=D_t(\bar G(x,t,[u]_\alpha)),
\end{equation}
then the quantities
\begin{equation}\label{xtclaw}
x^k\rho(t)(D_t^\alpha u-D_x^q K([u])), \qquad k=0,\ldots q-1,
\end{equation}
are conservation laws of (\ref{FPDE}).
\end{theorem}
\begin{proof}
The quantities in (\ref{xtclaw}) all vanish when $u$ is a solution of (\ref{FPDE}). Therefore, we only need to prove that these quantities can be written in the form (\ref{CLaw}).

Taking into account that $\rho$ and $\bar{G}$ satisfy (\ref{tclaw}), we obtain,
\begin{align*}
x^k\rho(t)(D_t^\alpha u-D_x^q K([u]))=&\, D_t(x^k\bar G(x,t,[u]_\alpha))-x^kD_x^q(\rho(t)K([u]))
\end{align*}
that, integrating by parts, can be written as a total divergence (\ref{CLaw}) with
\begin{equation}\label{exFG}
 F=\sum_{\ell=0}^k (-1)^{\ell+1}\frac{k!}{(k-\ell)!}x^{k-\ell}D_x^{q-1-\ell}(\rho(t)K([u])),\quad G=x^k\bar G,
\end{equation}
and $k=0,\ldots,q-1$.
\end{proof}
\begin{corollary}\label{coroll}
Equation (\ref{FPDE}) with $D_t^\alpha=\,^{\text{RL}}D_t^\alpha$ has at least $p\cdot q$ conservation laws given by (\ref{CLaw}) where $F$ and $G$ are defined by (\ref{exFG}) with $$\rho(t)=t^j,\qquad \bar G=\sum_{i=0}^{j}(-1)^i\frac{j!}{(j-i)!}t^{j-i}D_t^{p-i-1}(I_t^{p-\alpha}u),\qquad j=0,\ldots p-1. $$
\end{corollary}
\begin{proof}
It follows from the definition of Riemann-Liouville fractional derivative \eqref{RLdef} that equation (\ref{FPDE}) is itself a conservation law.
In fact, it can be written in the form (\ref{CLaw}) with
\begin{equation}\label{eqclaw}
F=-D_x^{q-1}K(u),\qquad G=D_t^{p-1}(I_t^{p-\alpha}u)=\left\{\begin{array}{cc}I_t^{1-\alpha}u,&\text{if}\quad p=1,\\
^{RL}D_t^{\alpha-1}u,&\text{if}\quad p>1.\end{array}\right.
\end{equation}
As a consequence, Theorem~\ref{Theocont} holds true with $\rho(t)=1$ and $\bar G=G$. Hence, equation (\ref{FPDE}) has at least $q$ conservation laws given in (\ref{xtclaw}). If $p=1$, the statement is proved. If $p>1$,
\begin{equation}
\rho(t)D_t^\alpha u:= t^jD_t^\alpha u=t^jD_t^{p}(I_t^{p-\alpha}u)=D_t(\bar G),\qquad j=1,\ldots, p-1,
\end{equation}
where the last equality is obtained after integrating by parts $j$ times with
\begin{equation}\label{Gbar}
\bar G=\sum_{i=0}^{j}(-1)^i\frac{j!}{(j-i)!}t^{j-i}D_t^{p-i-1}(I_t^{p-\alpha}u), \qquad j=1,\ldots,p-1.
\end{equation}
Therefore, it follows from  Theorem~\ref{Theocont} that for each $j=0,\ldots,p-1,$ there are $q$ conservation laws with flux and densities defined by
\begin{equation}\label{rhot}
F=\sum_{\ell=0}^k (-1)^{\ell+1}\frac{k!}{(k-\ell)!}x^{k-\ell}D_x^{q-1-\ell}(t^j K([u])),\quad G=x^k\bar G,\quad k=0,1,\ldots,q-1,
\end{equation}
with $\bar {G}$ given in (\ref{Gbar}), and so a total of at least $p\cdot q$ conservation laws.
\end{proof}
\begin{remark}
The function $\bar G$ in (\ref{Gbar}) can be equivalently written as
\begin{align*}
&\bar G=\sum_{i=0}^{j}(-1)^i\frac{j!}{(j-i)!}t^{j-i}D_t^{\alpha-i-1}u,& \text{\,\,if\,\,}j<p-1,\\
&\bar G=\sum_{i=0}^{p-2}(-1)^i\frac{(p-1)!}{(p-i-1)!}t^{p-i-1}D_t^{\alpha-i-1}u+(-1)^{p-1}{(p-1)!}I_t^{p-\alpha}u,& \text{\,\,if\,\,}j=p-1.
\end{align*}
\end{remark}

\begin{remark}
We observe that for $\alpha=p=1$ and $q=2$, (\ref{tclaw}) and (\ref{xtclaw}) with $\rho(t)=1$ are the two conservation laws given in \cite{Ibra}.
\end{remark}

\subsection*{Discrete setting}
In order to define a numerical approximation of (\ref{FPDE}) we define a uniform spatial grid with nodes,
\begin{equation}\label{spgrid}
x_i=a+i\Delta x,\qquad i=0,\ldots,M+1,\qquad \Delta x=\frac{b-a}{M+1}.
\end{equation}
Considering that at the endpoints the solution is known from the boundary conditions (\ref{Dir}), we define the vector of the approximations
\begin{equation}\label{Uvec}
\mathbf{u}=\mathbf{u}(t)\in\mathbb{R}^{M},\qquad \mathbf{u}_i(t)\simeq u(x_i,t),\quad i=1,\ldots,M.
\end{equation}
We denote with $D_{\Delta x}$ the forward difference operators in space and with $D^{(q)}_{\Delta x}$ the second-order centred difference operator for the $q$-th derivative.

We consider here semidiscretizations of the form
\begin{equation}\label{SDeq}
D_t^\alpha \mathbf{u}-D_{\Delta x}^{(q)}\widetilde{K}(\mathbf{u})=0,
\end{equation}
where $\widetilde{K}\approx K$ is here arbitrary, and it can be defined in such a way to obtain accuracy in space of arbitrary order. In fact, high-order finite difference approximations of the $q$-th derivative are defined on larger stencils and are obtained combining $D_{\Delta x}^{(q)}$ with suitable averaging operators (see e.g. formulae in \cite{Pier}).

Let be
\begin{equation}\label{tjnodes}
t_0<t_1<\ldots<t_{N-1}<t_N=T, \quad t_{j+1}=t_j+\Delta t_j, \quad j=0,\ldots,N-1,
\end{equation}
the nodes in time and $D_{\Delta t_j}$ the forward difference operator with step $\Delta t_j$. For simplicity of notation, henceforth we omit the subscript $j$ in the time difference operator. We denote with $u_{i,j}$ and $D_{\Delta t}^\alpha u_{i,j}$ the approximations of $u(x_i,t_j)$ and of $D_{t}^\alpha u(x_i,t_j)$, respectively, obtained after applying a suitable time integrator to (\ref{SDeq}). Hence, the fully discrete scheme for (\ref{FPDE}) is
\begin{equation}\label{spect}
D_{\Delta t}^\alpha u_{i,j}-D_{\Delta x}^{(q)} \widetilde{K}(u_{i,j})=0.
\end{equation}
Theorem~\ref{theodisc} will give sufficient conditions that method (\ref{spect}) has to satisfy to have discrete conservation laws in the form
$$D_{\Delta x}\widetilde{F}(x_i,t_j,u_{i,j})+D_{\Delta t}\widetilde{G}(x_i,t_j,u_{i,j})=0,$$
where $\widetilde{F}$ and $\widetilde{G}$ are suitable discretizations of the flux and density of a selected continuous conservation law, respectively.
This result is a discrete version of Theorem~\ref{Theocont}, and the main contribution in this section. For its proof, we recur to the following lemma that can be proved by straightforward calculations.
\begin{Lemma}\label{leibniz}
For any two discrete functions $f$ and $g$, the following discrete versions of Leibniz rule hold true:
\begin{align}\label{Leibdisc1}
-f_{i,j}D_{\Delta x}^{(1)}g_{i,j}=&\,g_{i,j}D_{\Delta x}^{(1)}f_{i,j}+D_{\Delta x}(-\tfrac{1}2(f_{i-1,j}g_{i,j}+f_{i,j}g_{i-1,j})),\\\label{Leibdisc2}
-f_{i,j}D_{\Delta x}^{(2)}g_{i,j}=&\,-g_{i,j}D_{\Delta x}^{(2)}f_{i,j}+D_{\Delta x}(\tfrac{1}{\Delta x}(f_{i,j}g_{i-1,j}-f_{i-1,j}g_{i,j})).
\end{align}
Moreover, for $q>2$,
\begin{equation*}
-f_{i,j}D_{\Delta x}^{(q)}g_{i,j}=-(-1)^q g_{i,j}D_{\Delta x}^{(q)}f_{i,j}+D_{\Delta x}(\widetilde{F}),\
\end{equation*}
where the function $\widetilde F$ is obtained by iterating (\ref{Leibdisc2}) $\lambda$ times, if $q=2\lambda$, or (\ref{Leibdisc2}) $\lambda$ times and (\ref{Leibdisc1}) once, if $q=2\lambda+1$.
\end{Lemma}
\begin{theorem}\label{theodisc}
For all $\rho(t_j)$ and $\bar{G}=\bar{G}(x_i,t_j,u_{i,j})$ such that
\begin{equation}\label{disctclaw}
\rho(t_j)D_{\Delta t}^\alpha u_{i,j}=D_{\Delta t}(\bar{G}(x_i,t_j,u_{i,j})),
\end{equation}
the quantities
\begin{equation}\label{discxtclaw}
x_i^k\rho(t_j)(D_{\Delta t}^\alpha u_{i,j}-D_{\Delta x}^{(q)} \widetilde K(u_{i,j})), \qquad k=0,\ldots q-1,
\end{equation}
are conservation laws of (\ref{spect}) at the point $(x_i,t_j)$ that approximate their continuous counterparts with the same accuracy of the method.
\end{theorem}
\begin{proof}
The proof follows along similar lines as that of Theorem~\ref{Theocont}. Multiplying method (\ref{spect}) by $x_i^k\rho(t_j)$, with $k=0,1,\ldots,q-1$, yields
\begin{equation}\label{discClaws}
x_i^k\rho(t_j)(D_{\Delta t}^\alpha u_{i,j}-D_{\Delta x}^{(q)}\widetilde{K}(u_{i,j})).
\end{equation}
These quantities clearly vanish on solutions of (\ref{spect}). Moreover, considering (\ref{disctclaw}) and Lemma~\ref{leibniz},
\begin{align*}
x_i^k\rho(t_j)&(D_{\Delta t}^\alpha u_{i,j}\,-D_{\Delta x}^{(q)} \widetilde{K}(u_{i,j}))= D_{\Delta t}(x_i^k\bar{G}(x_i,t_j,u_{i,j}))-x_i^kD_{\Delta x}^{(q)}(\rho(t_j)\widetilde{K}(u_{i,j}))\\
=&\,D_{\Delta t}(x_i^k\bar{G}(x_i,t_j,u_{i,j}))-((-1)^q D_{\Delta x}^{(q)}x_i^k)\rho(t_j)\widetilde K(u_{i,j})+D_{\Delta x}(\widetilde F(x_i,t_j,u_{i,j})).
\end{align*}
Then, as the spatial grid is uniform and $k<q$,
$$x_i^k\rho(t_j)(D_{\Delta t}^\alpha u_{i,j}-D_{\Delta x}^{(q)} \widetilde K(u_{i,j}))=D_{\Delta t}(\widetilde G)+D_{\Delta x}(\widetilde F),$$
where $\widetilde G=x_i^k\bar{G}$ and $\widetilde F \approx F$ in (\ref{exFG}). As the function $x_i^k\rho(t_j)$ is exactly evaluated at the nodes, a conservation law in the form (\ref{discxtclaw}) approximates its continuous limit (\ref{xtclaw}) with the same accuracy of the scheme.
\end{proof}

\section{The time integrator}\label{methsec}

Given the space discretization \eqref{SDeq}, we perform the time integration by using the spectral method introduced in \cite{bcdp17} for fractional problems of order $\alpha$ with $0<\alpha<1$. This method is here generalized to deal with equations of arbitrary fractional order. We separate the treatment of equations with fractional derivative satisfying Riemann-Liouville or Caputo definition with a focus on the particular case of zero initial conditions.

\subsection{Riemann-Liouville fractional derivative}\label{sec:RL}
When the fractional derivative $D_t^\alpha$ in (\ref{FPDE}) satisfies the Riemann-Liouville definition (\ref{RLdef}), we look for time approximations of the solution \eqref{Uvec} and of the initial conditions (\ref{IC}) at the node $x_i$, in the form
\begin{align}\label{modalexp}
u_N^i(t)=&\,\sum_{j=0}^{N+p} \hat u_j^i\mathcal{P}_j(t),\\
\label{discIC}
D_t^{\alpha-k} u_N^i(t_0)=&\,\sum_{j=0}^{N+p} \hat u_j^iD_t^{\alpha-k}\mathcal{P}_j(t_0)=\gamma_k(x_i),\qquad k=0,\ldots,p-1,
\end{align}
respectively, where $\{\mathcal P_j(t)\}_{j=0}^{N+p}$ is a suitable functional basis and $\hat u_j^i$ are unknown coefficients.

By considering the set of collocation points (\ref{tjnodes}), we can equivalently write equation (\ref{modalexp}) as
\begin{equation}\label{nodalexp}
u_N^i(t)=\sum_{k=0}^{N} \varphi_k(t)u_N^i(t_k)+\sum_{k=0}^{p-1}\varphi_{N+k+1}(t)\gamma_{k}(x_i),
\end{equation}
where the functions $\varphi_k(t)$ are unknown. By defining
$$\psi_k(t)=D_t^\alpha\varphi_k(t),\qquad k=0,\ldots,N+p,$$
expansion (\ref{nodalexp}) gives the following approximation of the time fractional derivative of $\mathbf{u}_i$ defined in \eqref{Uvec} at the collocation points:
\begin{equation}\label{fullapp}
D_t^\alpha \mathbf{u}_i(t_j)\approx D_t^\alpha  u_N^i(t_j)=\sum_{k=0}^{N}\psi_k(t_j)u_N^i(t_k)+\sum_{k=0}^{p-1}\psi_{N+k+1}(t_j)\gamma_k(x_i)=: D_{\Delta t}^\alpha u_{i,j}.
\end{equation}
In order to determine the functions $\psi_k(t)$, and so to practically compute approximation (\ref{fullapp}), we define the following two vectors in $\mathbb{R}^{N+p+1}$,
\begin{align*}
\uu_N^i&\,=\left(u_N^i(t_0),\ldots,u_N^i(t_N),\gamma_0(x_i),\ldots,\gamma_{p-1}(x_i)\right)^T,\\
\hat\uu^i&\,=\left(\hat u^i_0,\ldots,\hat u_{N+p}^i\right)^T,
\end{align*}
the matrices of dimension $N+p+1$,
\begin{equation}\label{matA}
A=\left(\begin{array}{c} A_1\\A_2\end{array}\right),\qquad B=A^{-1},
\end{equation}
with
\begin{equation*}
{A}_1=\left(\begin{array}{ccc}
\mathcal{P}_0(t_0) & \cdots & \mathcal{P}_{N+p}(t_0)\\
\vdots & & \vdots\\
\mathcal{P}_0(t_N) & \cdots & \mathcal{P}_{N+p}(t_N)
\end{array}\right), \
{A}_2=\left(\begin{array}{ccc}
D_t^\alpha \mathcal{P}_0(t_0) & \cdots & D_t^{\alpha} \mathcal{P}_{N+p}(t_0)\\
\vdots & & \vdots\\
D_t^{\alpha-p+1} \mathcal{P}_0(t_0) & \cdots & D_t^{\alpha-p+1} \mathcal{P}_{N+p}(t_0)\\
\end{array}\right).
\end{equation*}
and the matrices
\begin{align}\label{matC}
C=\left(\begin{array}{ccc}
\psi_0(t_0) & \cdots & \psi_{N+p}(t_0)\\
\vdots & & \vdots\\
\psi_0(t_N) & \cdots & \psi_{N+p}(t_N)
\end{array}\right), \
\mathcal{P}=\left(\begin{array}{ccc}
D_t^\alpha\mathcal{P}_0(t_0) & \cdots & D_t^\alpha\mathcal{P}_{N+p}(t_0)\\
\vdots & & \vdots\\
D_t^\alpha\mathcal{P}_0(t_N) & \cdots & D_t^\alpha\mathcal{P}_{N+p}(t_N)
\end{array}\right).
\end{align}
\begin{prop}
Matrix C can be computed as
\begin{equation}\label{matsys}
C=\mathcal PB.
\end{equation}
\end{prop}
\begin{proof}
By evaluating (\ref{modalexp}) at the nodes $t_j$, with $j=0,\ldots,N$, we obtain a set of $N+1$ equations that together with (\ref{discIC}) forms an algebraic system that can be equivalently written as
$$\uu_N^i=A\hat\uu^i,$$
therefore,
$$\hat\uu^i=B\uu_N^i,$$
or, entry-wise,
$$\hat u^i_j=\sum_{k=0}^N B_{j+1,k+1}u_N^i(t_k)+\sum_{k=0}^{p-1}B_{j+1,N+k+2}\gamma_{k}(x_i),\qquad j=0,\ldots,N+p.$$
Substituting in (\ref{modalexp}),
\begin{align}\nonumber
u_N^i(t)=&\,\sum_{j=0}^{N+p} \mathcal{P}_j(t)\left(\sum_{k=0}^N B_{j+1,k+1}u_N^i(t_k)+\sum_{k=0}^{p-1}B_{j+1,N+k+2}\gamma_k(x_i)\right)\\\label{tophi}
=&\,\sum_{k=0}^N\left(\sum_{j=0}^{N+p} B_{j+1,k+1}\mathcal{P}_j(t) \right)u_N^i(t_k)+\sum_{k=0}^{p-1}\left(\sum_{j=0}^{N+p}B_{j+1,N+k+2}\mathcal{P}_j(t) \right)\gamma_{k}(x_i).
\end{align}
Comparing (\ref{nodalexp}) and (\ref{tophi}), we find that
$$\varphi_k(t)=\sum_{j=0}^{N+p} B_{j+1,k+1}\mathcal{P}_j(t), \qquad k=0,\ldots, N+p,$$
and, differentiating,
\begin{equation}\label{psik}
\psi_k(t)=D_t^\alpha \varphi_k(t)=\sum_{j=0}^{N+p} B_{j+1,k+1}D_t^\alpha \mathcal{P}_j(t), \qquad k=0,\ldots, N+p.
\end{equation}
The values $\psi_k(t_j)$ are then obtained as
\begin{equation}\label{psiktj}
\psi_k(t_j)=\sum_{\ell=0}^{N+p} B_{\ell+1,k+1}D_t^\alpha \mathcal{P}_\ell(t_j), \quad k=0,\ldots, N+p, \ j=0,\ldots,N.
\end{equation}
Equation \eqref{matsys} follows immediately.
\end{proof}

Note that in order to compute $\psi_k(t_j)$ in (\ref{psiktj}) we only need the values of the chosen basis $\{\mathcal P_\ell(t)\}_{\ell=0}^{N+p}$ and its fractional derivatives at the collocation points. Therefore, the basis should be chosen such that it is easy to compute these fractional derivatives.

The fully discrete scheme for (\ref{FPDE}) is in the form (\ref{spect}) with the approximation of the fractional derivative (\ref{fullapp}) and
$$u_{i,j}:=u_N^i(t_j).$$
Let us consider matrix $C$ defined in (\ref{matC}) and let be $C=[C_1,C_2]$, with $C_1$ a square matrix of dimension $N+1$ and $C_2$ of dimension $ (N+1)\times p$. The numerical method can be written in matrix form as
\begin{equation}\label{matrixform}
C_1U=\frac{1}{\Delta x^2}(\widetilde{K}(U)\mathcal{M}+\mathcal{F})-C_2\gamma,
\end{equation}
where the entries of $U\in\mathbb{R}^{(N+1)\times M}$ and $\gamma \in\mathbb{R}^{p\times M}$ are
\begin{align*}
&U_{k,i}=u_{i,k-1}, \quad \gamma_{\ell,i}=\gamma_{\ell-1}(x_i),\quad k=1,\ldots,N+1, \ i=1,\ldots,M, \ \ell=1,\ldots p,
\end{align*}
respectively, matrix $\widetilde{K}(U)$ is obtained by applying $\widetilde{K}$ to the entries of matrix $U$, and
\begin{equation*}
\mathcal{M}=\left(\begin{array}{ccccc} -2 & 1 & && \\ 1 & -2 & 1 &&\\ &\ddots & \ddots& \ddots &\\ &&1 & -2 & 1 \\ && & 1 &-2\end{array}\right),\quad
\mathcal{F}=\left(\begin{array}{ccccc} \widetilde{K}(\chi_a(t_0)) & 0 &\ldots & 0 & \widetilde{K}(\chi_b(t_0))\\
\vdots & \vdots & &\vdots & \vdots\\
\widetilde{K}(\chi_a(t_N)) & 0 &\ldots & 0 & \widetilde{K}(\chi_b(t_N))\end{array}\right),
\end{equation*}
are matrices in $\mathbb{R}^{M\times M}$ and $\mathbb{R}^{(N+1)\times M},$ respectively.

\begin{remark}\label{remr}
In particular cases, the basis $\{\mathcal P_j(t)\}_j$ can be chosen such that $u_N^i$ satisfies $r\leq p$ initial conditions (\ref{IC}) by definition and do not need to be enforced.
In these cases, approximation (\ref{modalexp}) is taken in a projection space of dimension $N+p+1-r$. Similarly, approximation (\ref{nodalexp}) is replaced with
$$u_N^i(t)=\sum_{k=0}^{N} \varphi_k(t)u_N^i(t_k)+\sum_{k=0}^{p-1-r}\varphi_{N+k+1}(t)\gamma_{k}(x_i),$$
where (after a suitable reordering of the indexes) the second sum includes only the values $\gamma_k$ of the initial conditions to be imposed. Matrices $A$ and $B$ in \eqref{matA} have then reduced dimension $N+p+1-r$, and $\mathcal{P}$ in \eqref{matC} has dimension $(N+1)\times (N+p+1-r)$. System \eqref{matrixform} has still dimension $(N+1)\times M$ but matrices $C_2$ and $\gamma$ have dimension $(N+1)\times (p-r) $ and $(p-r)\times M$, respectively.
\end{remark}

\subsection{Caputo fractional derivative}\label{sec:Cap}
We now adapt the method introduced in Section~\ref{sec:RL} to problems in the form \eqref{FPDE} with fractional derivative satisfying Caputo's definition \eqref{Cdef}. In this case the initial conditions are given by \eqref{ICC}. In particular, the initial value of $u$ at $t=t_0$ is known, and in order to obtain a compatible system the approximated solution must be taken in a space of reduced dimension $N+p$.

Therefore, equations \eqref{modalexp} and \eqref{discIC} are replaced with
\begin{align}\label{modalexpC}
u_N^i(t)=&\,\sum_{j=0}^{N+p-1} \hat u_j^i\mathcal{P}_j(t),\\
\label{discICC}
D_t^{k} u_N^i(t_0)=&\,\sum_{j=0}^{N+p-1} \hat u_j^iD_t^{k}\mathcal{P}_j(t_0)=\gamma_k(x_i),\qquad k=0,\ldots,p-1,
\end{align}
respectively.
With collocation points,
$$t_1<\ldots<t_N=T,\qquad t_{j+1}=t_j+\Delta t_j, \qquad j=1,\ldots,N-1,$$
equation \eqref{modalexpC} is equivalently written as
\begin{equation}\label{nodalexpC}
u_N^i(t)=\sum_{k=1}^{N} \varphi_k(t)u_N^i(t_k)+\sum_{k=0}^{p-1}\varphi_{N+k+1}(t)\gamma_{k}(x_i),
\end{equation}
where functions $\varphi_k$ are unknown. Defining
$$\psi_k(t)=D_t^\alpha\varphi_k(t),\qquad k=1,\ldots,N+p,$$
the approximation of the fractional derivative is then given by
\begin{equation}\label{fullappC}
D_t^\alpha \mathbf{u}_i(t_j)\approx D_t^\alpha  u_N^i(t_j)=\sum_{k=1}^{N}\psi_k(t_j)u_N^i(t_k)+\sum_{k=0}^{p-1}\psi_{N+k+1}(t_j)\gamma_k(x_i)=: D_{\Delta t}^\alpha u_{i,j}.
\end{equation}
Following similar steps as in Section \ref{sec:RL}, the values
$$\psi_k(t_j),\qquad k=1,\ldots, N+p, \qquad j=1,\ldots,N,$$
can be obtained by solving \eqref{matsys} where matrices $C$ and $\mathcal{P}$ have now dimension $N\times (N+p)$ and are defined as their analogues in \eqref{matC} by deleting the first row and the first column and the first row and the last column, respectively. Matrix $B$ has dimension $(N+p)\times (N+p)$ and is defined as in \eqref{matA} after removing from $A$ the first row and the last column.

By splitting matrix $C$ as $C=[C_1,C_2]$ where $C_1$ and $C_2$ have dimension $N\times N$ and $N\times p$, respectively, and defining $U\in\mathbb{R}^{N\times M}$ with entries $U_{k,i}=u_N^i(t_k),$ $k=1,\ldots,N$, $i=1,\ldots,M$, the numerical method can be written in matrix form as
\begin{equation}\label{eq:Cmatform}
C_1 U=\frac{1}{\Delta x^2}(\widetilde{K}(U)\mathcal{M}+\mathcal{F})-C_2\gamma,
\end{equation}
where $\mathcal{M}$, $\gamma$ and $\widetilde{K}(U)$ are defined as in \eqref{matrixform} and \begin{equation*}
\mathcal{F}=\left(\begin{array}{ccccc} \widetilde{K}(\chi_a(t_1)) & 0 &\ldots & 0 & \widetilde{K}(\chi_b(t_1))\\
\vdots & \vdots & &\vdots & \vdots\\
\widetilde{K}(\chi_a(t_N)) & 0 &\ldots & 0 & \widetilde{K}(\chi_b(t_N))\end{array}\right)\in\mathbb{R}^{N\times M}.
\end{equation*}
We observe that the dimension of system \eqref{eq:Cmatform} is lower than that of system \eqref{matrixform}.

\begin{remark}\label{rem:cap}
If the basis $\{\mathcal P_j(t)\}_j$ is chosen such that $r\leq p$ initial conditions (\ref{IC}) are satisfied by $u_N^i$ an analogue discussion as in Remark~\ref{remr} holds considering a projection space of dimension $N+p-r$.
\end{remark}

\subsubsection{The case of zero initial conditions}
When the initial condition is of total rest, i.e.,
\begin{equation}\label{IC2}
D_t^k u(x,t_0)=0,\qquad k=0,\ldots p-1,
\end{equation}
the definitions of Caputo and Riemann-Liouville fractional derivative are equivalent.
Since in this special case $\gamma_k(x_i)=0$, approximation (\ref{fullappC}) reduces to
\begin{equation*}
D_t^\alpha \mathbf{u}_i(t_j)\approx D_t^\alpha  u_N^i(t_j)=\sum_{k=1}^{N}\psi_k(t_j)u_N^i(t_k)=: D_{\Delta t}^\alpha u_{i,j},
\end{equation*}
and so only the values of $\psi_k(t_j)$ for $k=1\ldots,N,$ and $j=1\ldots,N$ need to be calculated.
The matrix form of the numerical method \eqref{eq:Cmatform} reduces to
\begin{equation}\label{eq:matform0}
C_1 U=\frac{1}{\Delta x^2}(\widetilde{K}(U)\mathcal{M}+\mathcal{F}),
\end{equation}
and so it is not necessary to calculate matrix $C_2$. Matrix $C_1$ is obtained from
$$C_1={\mathcal P} \bar B,$$
where $\bar B$ is obtained removing from matrix $B$ the last $p$ columns.

Note that although the computation of the inverse of matrix $A$, having dimension $N+p$, is still required, system \eqref{eq:matform0} has reduced dimension, $N$.

\begin{remark}
The basis
\begin{equation}\label{eq:powerb}
\{\mathcal{P}_j(t)\}=\{(t-t_0)^{j\alpha}\},
\end{equation} identically satisfies all the initial conditions (\ref{IC2}) for $k=1,\ldots,p-1,$ and so Remark \ref{rem:cap} applies. In this case, then, matrix $A$ has dimension $N+1$.
\end{remark}

\section{Conservation laws of the fractional diffusion equation}\label{secmodel}
Here and henceforth we consider equation (\ref{FPDE}) with $q=2$ and Riemann-Liouville fractional derivative.

Although, as seen in Section~\ref{methsec}, the general discussion holds for methods of arbitrary order in space, in order to give some specific results and explicit formulae of the preserved conservation laws, we focus here on second-order accurate schemes. Therefore, henceforth we set $\widetilde{K}=K$ and method (\ref{spect}) reduces to
\begin{equation}\label{spectK}
D_{\Delta t}^\alpha u_{i,j}-D_{\Delta x}^{(2)} {K}(u_{i,j})=0.
\end{equation}
As stated in Corollary~\ref{coroll}, when the fractional derivative is defined according to the definition of Riemann-Liouville there are always at least two conservation laws.
The first is equivalent to equation (\ref{FPDE}) and it is defined by (see \eqref{eqclaw})
\begin{equation}\label{SDclaw1}
F_1(x,t,u)=-D_x K(u),\qquad G_1(x,t,u)=D_t^{p-1}I_t^{p-\alpha}u.
\end{equation}
The second is defined by (\ref{rhot}) with $k=1$ and $\bar{G}=G_1$, yielding,
\begin{equation}\label{SDclaw2}
F_2(x,t,u)= K(u)-xD_xK(u),\qquad G_2(x,t,u)=x G_1.
\end{equation}
We prove that method (\ref{spectK}) preserves these conservation laws for any value of $\alpha$ and $p$, and give their conserved approximations. According to Theorem~\ref{theodisc}, it suffices to prove (\ref{disctclaw}), i.e. to find $\widetilde{G}_1\approx G_1$ such that
\begin{equation}\label{cond}
D_{\Delta t}^\alpha u_{i,j}=D_{\Delta t}(\widetilde{G}_1(x_i,t_j,u_{i,j})).
\end{equation}
By integrating  both sides in (\ref{RLdef}), we have
$$G_1(x,t,u)=D_t^{p-1}I_{t}^{p-\alpha}u(x,t)=\int_{t_0}^{t}D_\tau^\alpha u(x,\tau)\,\mathrm{d}\tau.$$
We define then
\begin{equation}\label{appG1}
\widetilde G_1(x_i,t_{j},u_{i,j})=\Delta t \sum_{\ell=0}^{j-1}  D_{\Delta t}^{\alpha}u_{i,\ell},\end{equation}
so that (\ref{cond}) is satisfied. Following the steps in the proof of Theorem~\ref{theodisc}, the remaining functions that define the two discrete conservation laws are
\begin{equation}\label{discF1}
\widetilde{F}_1(x_i,t_{j},u_{i,j})=-D_{\Delta x}{K}(u_{i-1,j}),
\end{equation}
and (see (\ref{Leibdisc2})),
\begin{align}\nonumber
\widetilde{F}_2(x_i,t_{j},u_{i,j})=& \displaystyle \,\tfrac{1}{\Delta x}(x_iK(u_{i-1,j})-x_{i-1}K(u_{i,j}))\\
\nonumber =& A_{\Delta x}{K}(u_{i-1,j})\-A_{\Delta x}(x_{i-1})D_{\Delta x}K(u_{i-1,j}),\\
\label{discF2G2}\widetilde{G}_2(x_i,t_{j},u_{i,j})=&\,x_i\widetilde{G}_1(x_i,t_{j},u_{i,j}),
\end{align}
respectively, where $A_{\Delta x}$ denotes the forward average operator,
$$A_{\Delta x}(f(u_{i,j}))=\frac{f(u_{i+1,j})+f(u_{i,j})}2=f(u(x_i+\tfrac{\Delta x}2,t_j))+\mathcal{O}(\Delta x^2).$$
In the rest of this section, we focus on the two cases of subdiffusion and superdiffusion equations obtained setting $p=1$ and $p=2$, respectively.
The continuous and discrete conservation laws obtained are listed in Table~\ref{allcl}.
\subsubsection*{Subdiffusion-wave equation}
When $0<\alpha<1=p$ equation \eqref{FPDE} defines a subdiffusion problem. Using Theorem~\ref{Theocont} and Corollary~\ref{coroll} we can obtain only the two conservation laws defined by (\ref{SDclaw1}) and (\ref{SDclaw2}). Indeed, these two are the only independent conservation laws of the subdiffusion-wave equation in the generic form (\ref{FPDE}) with Riemann-Liouville fractional derivative \cite{Luk}. As shown, method (\ref{spectK}) has discrete analogues of this conservation laws defined by \eqref{appG1}--\eqref{discF2G2}, for any $p$. Some extra conservation laws are given in \cite{Luk} for special choices of the function $K$ in (\ref{FPDE}), but these depend on integrals whose order is larger than $p-\alpha$ that is a case that we do not consider in this paper.
\subsubsection*{Superdiffusion equation}
Let us consider now the superdiffusion equation defined by (\ref{FPDE}) with $1<\alpha<2=p$. In this case, according to Corollary \ref{coroll}, conservation laws are obtained from (\ref{xtclaw}) with
$$\rho(t)=1\qquad \rho(t)=t,$$
yielding four independent conservation laws. The first two are again (\ref{SDclaw1}) and (\ref{SDclaw2}) and these are preserved by method (\ref{spectK}). The other two conservation laws are defined with (see \eqref{Gbar}--(\ref{rhot})),
\begin{equation}\label{FG3}
F_3(x,t,u)=-tD_xK(u),\qquad G_3(x,t,u)=tD_t^{\alpha-1}u-I_t^{2-\alpha}u,
\end{equation}
and
\begin{equation}\label{FG4}
F_4(x,t,u)=tK(u)-xtD_xK(u),\qquad G_4(x,t,u)=xtD_t^{\alpha-1}u-xI_t^{2-\alpha}u.
\end{equation}
The density function in (\ref{FG3}) can be equivalently written as,
$$G_3(x,t,u)=tG_1(x,t,u)-\int_{t_0}^tG_1(x,z,u)\,\mathrm{d}z.$$
In fact, for $p=2$,
$$D_t^{\alpha-1}u=D_tI_t^{2-\alpha}u=G_1(x,t,u),$$
and integrating twice the Riemann-Liouville fractional derivative (\ref{RLdef}) yields,
$$I_t^{2-\alpha}u=\int_{t_0}^t\int_{t_0}^z D_\tau^\alpha u(x,\tau)\,\mathrm{d}\tau\,\mathrm{d}z=\int_{t_0}^tG_1(x,z,u)\,\mathrm{d}z.$$
We show that (\ref{disctclaw}) is satisfied by the solutions of (\ref{spectK}) with
\begin{equation}\label{discG3}
\rho(t_j)=t_j,\quad \bar{G}=\widetilde{G}_3(x_i,t_j,u_{i,j})=t_j\widetilde{G}_1(x_i,t_j,u_{i,j})-\Delta t\sum_{r=0}^{j} \widetilde G_1(x_i,t_r,u_{i,r}),
\end{equation}
where $\widetilde{G}_1(x_i,t_j,u_{i,j})$ is given in (\ref{appG1}).
In fact, equation (\ref{cond}) yields
\begin{align*}
t_jD_{\Delta t}^\alpha u_{i,j}&\,=t_jD_{\Delta t}\widetilde G_1(x_i,t_{j},u_{i,j})=D_{\Delta t}(t_j \widetilde G_1(x_i,t_{j},u_{i,j}))-\widetilde G_1(x_i,t_{j+1},u_{i,j+1})\\
&\,=D_{\Delta t}(\widetilde G_3(x_i,t_{j},u_{i,j})).
\end{align*}
Therefore, it follows from Theorem~\ref{theodisc} that method (\ref{spectK}) has other two conservation laws that approximate (\ref{FG3}) and (\ref{FG4}), and that are defined by $\widetilde{G}_3$ in  (\ref{discG3}), and
\begin{align*}
&\widetilde{F}_3(x_i,t_j,u_{i,j})=-t_jD_{\Delta x} ({K}(u_{i,j})),\qquad \widetilde{F}_4(x_i,t_{j},u_{i,j})=t_j\widetilde{F}_2(x_i,t_{j},u_{i,j}),\\& \widetilde{G}_4(x_i,t_{j},u_{i,j})=x_i\widetilde{G}_3(x_i,t_{j},u_{i,j}).
\end{align*}

\begin{table}[t]
\caption{Continuous and discrete conservation laws. Subdiffusion: $\ell=1,2$. Superdiffusion: $\ell=1,\ldots,4$.}\label{allcl}
\small
\begingroup
\setlength{\tabcolsep}{5pt}
\renewcommand{\arraystretch}{1.35}
\centerline{\begin{tabular}{|c||c|c||c|c|}
\hline
$ \ell$& $F_\ell$ & $G_\ell$ & $\widetilde F_\ell$ & $\widetilde G_\ell$ \\
\hline
{$1$} & $-D_xK(u)$ & $\int D_t^\alpha u\,\mathrm{d}t$& $-D_{\Delta x}K(u_{i-1,j})$ & $\Delta t \sum_{\ell=0}^{j-1} D_{\Delta t}^\alpha u_{i,\ell}$\\
\hline
$2$ & $K(u)\!-\!xD_xK(u)$ & $xG_1$& $A_{\Delta x}K(u_{i-1,j})\!-\!A_{\Delta x}(x_{i-1})D_{\Delta x}K(u_{i-1,j})$ & $x_i\widetilde{G}_1$\\
\hline
{$3$} & $-tF_1$ & $tG_1-\int G_1\,\mathrm{d}t$& $-t_j\widetilde{F}_1$ & $t_j\widetilde{G}_1-\Delta t \sum_{r=0}^{j} \widetilde{G}_1$\\
\hline
$4$ & $tF_2$ & $xG_3$& $t_j\widetilde{F}_2$ & $x_i\widetilde{G}_3$\\
\hline
\end{tabular}}
\endgroup
\end{table}

\section{Numerical tests}\label{sec:tests}
In this section we solve problem \eqref{FPDE} with two different choices of function $K$ that define a linear and a nonlinear problem, respectively. In both cases we set $q=2$ and $(x,t)\in (0,1)\times (0,2)$, and we study the two cases of subdiffusion ($p=1$) and superdiffusion ($p=2$). We consider the boundary conditions
\begin{equation}\label{eq:bctest}
u(0,t)=u(1,t)=\frac{t^p}p\qquad t\in[0,2],
\end{equation}
and an initial configuration of total rest \eqref{IC2}, i.e. if $0<\alpha<1$ (subdiffusion case),
\begin{equation}\label{eq:test_sub}
  u(x,0)=0,
\end{equation}
while, if $1<\alpha<2$ (superdiffusion case),
\begin{equation}\label{eq:test_sub2}
  u(x,0)=\frac{\partial }{\partial t} u(x,0)=0.
\end{equation}
Therefore, the fractional derivative in \eqref{FPDE} is equivalently defined by either equation \eqref{RLdef} or equation \eqref{Cdef}.

The space grid is defined as in (\ref{spgrid}) with $\Delta x=0.005$, that is small enough to study the rate of convergence in time of the method. Inspired by \cite{bcdp17} we choose the time nodes $t_j$ equal to the Chebyshev nodes in $[0,2]$ and the basis $\mathcal{P}_j$ in \eqref{modalexp}--\eqref{discIC} defined by the Jacobi polynomials in $[0,2]$ (see \cite{BHRAWY2015876}),
\begin{equation*}
\mathcal{P}_j(t)=\sum_{k=0}^j\frac{(-1)^{j+k}(j+k+p)!}{(k+p)!(j-k)!k!2^k}t^k,
\qquad j=0,\ldots, N+p.
\end{equation*}
Different choices, such as uniform nodes and the power basis \eqref{eq:powerb}, can also be considered. However, in the experiments below, choosing Chebyshev nodes and the Jacobi basis yields a matrix $C_1$ in \eqref{eq:matform0} with lower condition number and a faster rate of convergence, respectively.
\subsection*{Linear problem}
We consider here the linear fractional PDE defined by \eqref{FPDE} with $K(u)=u$, i.e.
\begin{equation}\label{eq:lineq}
D_t^\alpha u-D_x^2 u=0.
\end{equation}
The exact solution satisfying the boundary conditions \eqref{eq:bctest} and the initial conditions \eqref{eq:test_sub} or \eqref{eq:test_sub2} are given in \cite{materials} and amounts to
\begin{align}\label{uex}
u_{\text{exact}}(x,t)=&\,\sum_{n=1}^\infty a_N(x,t)
\\\nonumber
=&\,\sum_{n=1}^\infty \frac{-4t^p}{(2n-1)\pi}\sin{\left((2n-1)\pi x\right)}E_{\alpha,p+1}(-(2n-1)^2\pi^2t^\alpha)+\frac{t^p}p,
\end{align}
where
$$E_{\alpha,\beta}(z)=\sum_{k=0}^\infty \frac{z^k}{\Gamma(\alpha k+\beta)},$$
is the Mittag-Leffler function with two arguments. A reference solution, $\bar{u}$, is calculated by truncating the infinite sum in (\ref{uex}) after $R$ terms, where $R$ is the smallest integer such that $\|a_R\|<\text{tol} = 10^{-12}$, and by computing the Mittag-Leffler function using the \textsc{Matlab} routine $\texttt{ml}$ \cite{gar15,Garrappa2015}.

In order to show the convergence of the proposed numerical scheme, its spectral accuracy and its conservative properties, the error in the numerical solution and in the discrete conservation laws are calculated as
$$\text{Sol err}=\max_i\max_j |u_{i,j}-\bar u(x_i,t_j)|,$$
and
\begin{equation}\label{CLerr}
\text{Err}_\ell=\max_i\max_j |D_{\Delta t}\widetilde{G}_\ell(x_i,t_j,u_{i,j})+D_{\Delta x}\widetilde{F}_\ell(x_i,t_j,u_{i,j})|, \ \ell=1,\ldots,p\cdot q,
\end{equation}
with functions $\widetilde F_\ell$ and $\widetilde G_\ell$ defined in Section~\ref{secmodel}.

\begin{table}[t]
\caption{Linear problem. Errors in solution and conservation laws. \label{tab:CLaws}}
\small
\begingroup
\setlength{\tabcolsep}{6pt}
\renewcommand{\arraystretch}{1.12}
\centerline{\begin{tabular}{|c||c|c|c|c|c|}
\hline
$\alpha$ & Sol err & Err$_1$ & Err$_2$ & Err$_3$ & Err$_4$\\
\hline
0.1& 2.37e-06& 1.54e-11& 1.21e-11& N.A.& N.A.\\
0.5& 4.29e-04& 1.59e-11& 1.35e-11& N.A.& N.A.\\
0.9& 1.04e-03& 1.57e-11& 1.41e-11& N.A.& N.A.\\
1.1& 2.99e-05& 1.60e-11& 1.30e-11& 3.02e-11& 2.43e-11\\
1.5& 1.11e-04& 1.67e-11& 1.58e-11& 3.23e-11& 2.20e-11\\
1.9& 2.34e-03& 1.69e-11& 1.52e-11& 3.30e-11& 2.97e-11\\
\hline
\end{tabular}}
\endgroup
\end{table}
In Table~\ref{tab:CLaws} we show the error in the solution and conservation laws given by method \eqref{spectK} with $N=10$ applied to \eqref{eq:lineq}. We consider three different values of $\alpha$ corresponding to subdiffusion problems. Two of these values are close to the integer cases ($\alpha=0.1$ and $\alpha=0.9$). The third is an intermediate value ($\alpha=0.5$). Similarly, we consider three different superdiffusion problems ($\alpha=1.1,1.5,1.9$). The error in the solution is larger for values of $\alpha$ that are closer to $p$. In all cases the error in the conservation laws is only due to the roundoffs. In particular, we verified that the errors in the conservation laws are comparable in magnitude to the residual of equation \eqref{eq:matform0}.
\begin{figure}[p]
\begin{center}
\includegraphics[scale=0.7]{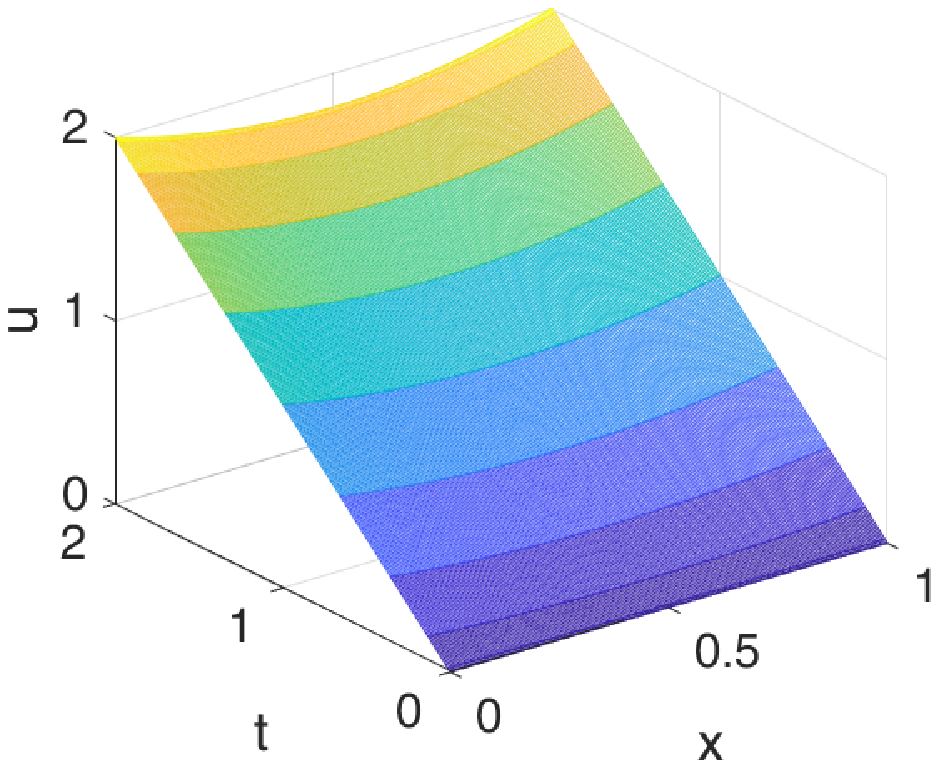}
\includegraphics[scale=0.7]{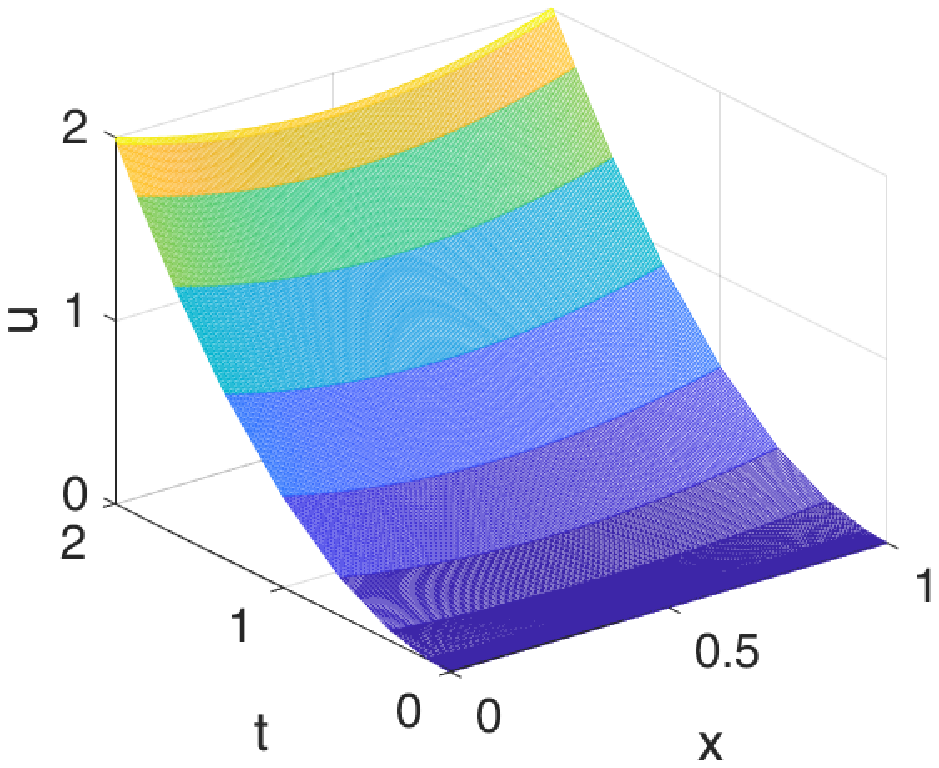}
\caption{Linear problem. Numerical solution with $N=10$, $\Delta x=0.005$, $\alpha=0.5$ (left) and $\alpha=1.5$ (right).}\label{diff3d}
\end{center}
\end{figure}
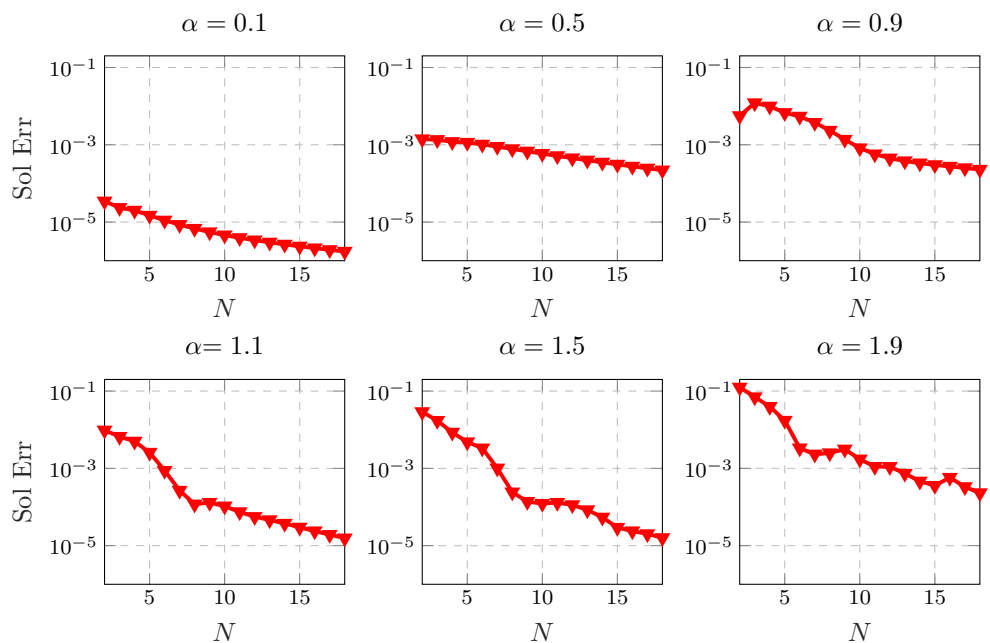
\begin{figure}[p]
\begin{center}
	\begin{tikzpicture}

\begin{axis}[%
width=1.257in,
height=1.073in,
at={(1.133in,2.333in)},
scale only axis,
xmin=2,
xmax=18,
xlabel style={font=\color{white!15!black}},
xlabel={$N$},
ymode=log,
ymin=1e-06,
ymax=0.2,
ylabel style={font=\color{white!15!black}},
ylabel={Sol Err},
yminorticks=true,
title={$\alpha=0.1$},
axis background/.style={fill=white},
xmajorgrids,
ymajorgrids,
yminorgrids,
legend style={legend cell align=left, align=left, draw=white!15!black}
]
\addplot [color=red, line width=1.5pt, mark=triangle, mark options={solid, rotate=180, red}]
  table[row sep=crcr]{%
1	0.000126614280401727\\
2	3.45415757511125e-05\\
3	2.35473671152953e-05\\
4	2.00168692499059e-05\\
5	1.48314040233632e-05\\
6	1.10974549831777e-05\\
7	8.55451335185498e-06\\
8	6.7879893145828e-06\\
9	5.52151076960161e-06\\
10	4.58529037188071e-06\\
11	3.90898216390267e-06\\
12	3.40369749085578e-06\\
13	2.98904197848504e-06\\
14	2.64574383745941e-06\\
15	2.35887663788126e-06\\
16	2.11697860846849e-06\\
17	1.91123626758216e-06\\
18	1.73483168659783e-06\\
19	1.5824658945035e-06\\
20	9.10124145736546e-06\\
21	5.11841650230327e-05\\
22	0.000250626464884274\\
};

\end{axis}
\end{tikzpicture}%
	\begin{tikzpicture}
\begin{axis}[%
width=1.257in,
height=1.073in,
at={(1.133in,2.333in)},
scale only axis,
xmin=2,
xmax=18,
xlabel style={font=\color{white!15!black}},
xlabel={$N$},
ymode=log,
ymin=1e-6,
ymax=0.2,
yminorticks=true,
title={$\alpha=0.5$},
axis background/.style={fill=white},
xmajorgrids,
ymajorgrids,
yminorgrids,
legend style={legend cell align=left, align=left, draw=white!15!black}
]
\addplot [color=red, line width=1.5pt, mark=triangle, mark options={solid, rotate=180, red}]
  table[row sep=crcr]{%
1	0.00284065383072019\\
2	0.00140850373306567\\
3	0.00134785154976022\\
4	0.00119614320704808\\
5	0.0011450841361661\\
6	0.00102523515871022\\
7	0.000898088664089469\\
8	0.00078127141531481\\
9	0.000678969109732683\\
10	0.000591023381394273\\
11	0.000515935688651978\\
12	0.000451923453577849\\
13	0.000397291889471678\\
14	0.000350550410224561\\
15	0.000310431739721617\\
16	0.00027587540356629\\
17	0.000245999978006211\\
18	0.000220075786014757\\
19	0.000197495198370517\\
20	0.000177747349259934\\
21	0.000217812126870776\\
22	0.000630712902265573\\
23	0.00114715891489348\\
};

\end{axis}
\end{tikzpicture}%
	\begin{tikzpicture}

\begin{axis}[%
width=1.257in,
height=1.073in,
at={(1.133in,2.333in)},
scale only axis,
xmin=2,
xmax=18,
xlabel style={font=\color{white!15!black}},
xlabel={$N$},
ymode=log,
ymin=1e-6,
ymax=0.2,
yminorticks=true,
axis background/.style={fill=white},
title={$\alpha=0.9$},
xmajorgrids,
ymajorgrids,
yminorgrids,
legend style={legend cell align=left, align=left, draw=white!15!black}
]
\addplot [color=red, line width=1.5pt, mark=triangle, mark options={solid, rotate=180, red}]
  table[row sep=crcr]{%
1	0.00607999101163892\\
2	0.00559637418455838\\
3	0.0120781386835688\\
4	0.00991748256853418\\
5	0.00672950475530926\\
6	0.00542300325328182\\
7	0.00374716355049362\\
8	0.00234357759249866\\
9	0.00136945310623933\\
10	0.000823715419667359\\
11	0.000573813519897624\\
12	0.000448961578712291\\
13	0.000379232351667154\\
14	0.000334677001583614\\
15	0.000301945965962325\\
16	0.000275022592513333\\
17	0.000251176152208637\\
18	0.00022932940711588\\
19	0.000209102971024507\\
20	0.000190409954648633\\
21	0.000173262561758012\\
22	0.000596940214391051\\
23	0.00372068980717044\\
24	0.0442779777149838\\
};

\end{axis}

\end{tikzpicture}%
\\
	\begin{tikzpicture}

\begin{axis}[%
width=1.257in,
height=1.073in,
at={(1.133in,2.333in)},
scale only axis,
xmin=2,
xmax=18,
xlabel style={font=\color{white!15!black}},
xlabel={$N$},
ymode=log,
ymin=1e-06,
ymax=0.2,
ylabel style={font=\color{white!15!black}},
ylabel={Sol Err},
yminorticks=true,
title={$\alpha{=1.1}$},
axis background/.style={fill=white},
xmajorgrids,
ymajorgrids,
yminorgrids,
legend style={legend cell align=left, align=left, draw=white!15!black}
]
\addplot [color=red, line width=1.5pt, mark=triangle, mark options={solid, rotate=180, red}]
  table[row sep=crcr]{%
1	0.0124367701663861\\
2	0.00968727866827357\\
3	0.00663650640995894\\
4	0.00500030672098367\\
5	0.00257698294939912\\
6	0.00087771211928378\\
7	0.000267808975999539\\
8	0.000118128735312917\\
9	0.000130327497141314\\
10	0.000105080443903283\\
11	7.45853638306482e-05\\
12	5.60175456783552e-05\\
13	4.62924863727787e-05\\
14	3.74908485689662e-05\\
15	2.97953337137224e-05\\
16	2.38205921757606e-05\\
17	1.92481158178825e-05\\
18	1.56418945050363e-05\\
19	1.27971084438738e-05\\
20	1.73497964146563e-05\\
21	8.53090368369003e-05\\
22	0.000372358014306684\\
};

\end{axis}

\end{tikzpicture}%
	\begin{tikzpicture}

\begin{axis}[%
width=1.257in,
height=1.073in,
at={(1.133in,2.333in)},
scale only axis,
xmin=2,
xmax=18,
xlabel style={font=\color{white!15!black}},
xlabel={$N$},
ymode=log,
ymin=1e-06,
ymax=0.2,
yminorticks=true,
title={$\alpha=1.5$},
axis background/.style={fill=white},
xmajorgrids,
ymajorgrids,
yminorgrids,
legend style={legend cell align=left, align=left, draw=white!15!black}
]
\addplot [color=red, line width=1.5pt, mark=triangle, mark options={solid, rotate=180, red}]
  table[row sep=crcr]{%
1	0.0147834914765592\\
2	0.0289993245906055\\
3	0.0172082048131172\\
4	0.00854096625975798\\
5	0.00480076329846746\\
6	0.00333182717321406\\
7	0.0010074825177302\\
8	0.000241933833714195\\
9	0.000138180690694212\\
10	0.000122754167721345\\
11	0.000130079586263411\\
12	0.00011317006898659\\
13	8.40091731469927e-05\\
14	5.40742703839525e-05\\
15	2.94110060136132e-05\\
16	2.38676117550937e-05\\
17	2.0119806586727e-05\\
18	1.59632206409854e-05\\
19	1.46526285061815e-05\\
20	1.36640325270054e-05\\
21	2.5830760520984e-05\\
22	5.97468278618063e-05\\
};

\end{axis}
\end{tikzpicture}%
	\begin{tikzpicture}
\begin{axis}[%
width=1.257in,
height=1.073in,
at={(1.133in,2.333in)},
scale only axis,
xmin=2,
xmax=18,
xlabel style={font=\color{white!15!black}},
xlabel={$N$},
ymode=log,
ymin=1e-6,
ymax=0.2,
yminorticks=true,
axis background/.style={fill=white},
title={$\alpha=1.9$},
xmajorgrids,
ymajorgrids,
yminorgrids,
legend style={legend cell align=left, align=left, draw=white!15!black}
]
\addplot [color=red, line width=1.5pt, mark=triangle, mark options={solid, rotate=180, red}]
  table[row sep=crcr]{%
1	0.0599352571016074\\
2	0.125038784196653\\
3	0.0702678643773402\\
4	0.0396474648388052\\
5	0.0173196738837789\\
6	0.00333539778343899\\
7	0.00231980856672379\\
8	0.00249473291666538\\
9	0.00312315623537018\\
10	0.00172403035581037\\
11	0.0011113691323496\\
12	0.00110454636281188\\
13	0.000737057392528273\\
14	0.000463086764053822\\
15	0.000358345541120672\\
16	0.000578554166168255\\
17	0.000332041063259014\\
18	0.000231476118079311\\
19	0.000262249717185803\\
20	0.000215068136356411\\
21	0.000211268427794398\\
22	0.00129115317139283\\
};

\end{axis}

\end{tikzpicture}%
\end{center}
	\caption{Linear problem. Rate of convergence in time. $\Delta x=0.005$ (logarithmic scale on $y$-axis).}
	\label{a2}
\end{figure}

In Fig.~\ref{diff3d} we show the solutions of method \eqref{spectK} with $\alpha=0.5$ and $\alpha=1.5$. These two graphs very well reproduce the behaviour of the exact solution shown in \cite{materials}.

In Figure~\ref{a2} we study the convergence of the method by plotting the logarithm of the error in the solution against $N$ for $N=2,\ldots,18$. These graphs show that the convergence of the method is exponential in time.

\subsection*{Nonlinear problem}
We consider now equation \eqref{FPDE} with $K(u)=\sqrt{u}$, therefore we solve equation
\begin{equation}\label{eq:sqrteq}
D_t^\alpha u-D_x^2 (\sqrt{u})=0,
\end{equation}
with boundary conditions \eqref{eq:bctest} and initial conditions \eqref{eq:test_sub} if $0< \alpha <1$ or \eqref{eq:test_sub2} if $1<\alpha <2$.

The exact solution of this problem is not known, and so we compute a reference solution, $\bar{u}$, setting $\bar N=13$. For $N<\bar N$, the error in the solution at the final time is estimated by
$$\text{Sol err}=\sqrt{\sum_i |u_{i,N}-\bar u_{i,\bar N}|}.$$
The error in the conservation laws is evaluated as in (\ref{CLerr}).

\begin{table}[t]
\caption{Nonlinear problem. Errors in solution and conservation laws.\label{tab:CLawsNL}}
\small
\begingroup
\setlength{\tabcolsep}{6pt}
\renewcommand{\arraystretch}{1.12}
\centerline{\begin{tabular}{|c||c|c|c|c|c|}
\hline
$\alpha$ & Sol err & Err$_1$ & Err$_2$ & Err$_3$ & Err$_4$\\
\hline
0.1& 6.78e-07 & 2.39e-11 & 1.99e-11 & N.A. & N.A.\\
0.5& 4.08e-15 & 2.46e-11 & 1.88e-11 & N.A. & N.A.\\
0.9& 4.37e-05 & 2.76e-11 & 2.25e-11 & N.A.& N.A.\\
1.1& 1.10e-05 & 2.33e-11 & 2.00e-11 & 4.35e-11 & 3.78e-11 \\
1.5& 2.00e-04 & 2.94e-11 & 2.83e-11 & 4.83e-11 & 4.64e-11 \\
1.9& 4.30e-03 & 2.89e-11 & 2.82e-11 & 5.33e-11 & 4.64e-11\\
\hline
\end{tabular}}
\endgroup
\end{table}

Table~\ref{tab:CLawsNL} shows the solution error and the error in the conservation laws given by method \eqref{spectK} with $N=10$ applied to \eqref{eq:sqrteq}. The error in the conservation laws is only due to the accuracy in the Newton method to solve the nonlinear system \ref{eq:matform0}. As in the linear case, the error in the solution is larger for $\alpha$ close to $p$, except that in this case the method exactly solves the subdiffusion problem with $\alpha=0.5$.

\begin{figure}[p]
\begin{center}
\includegraphics[scale=0.7]{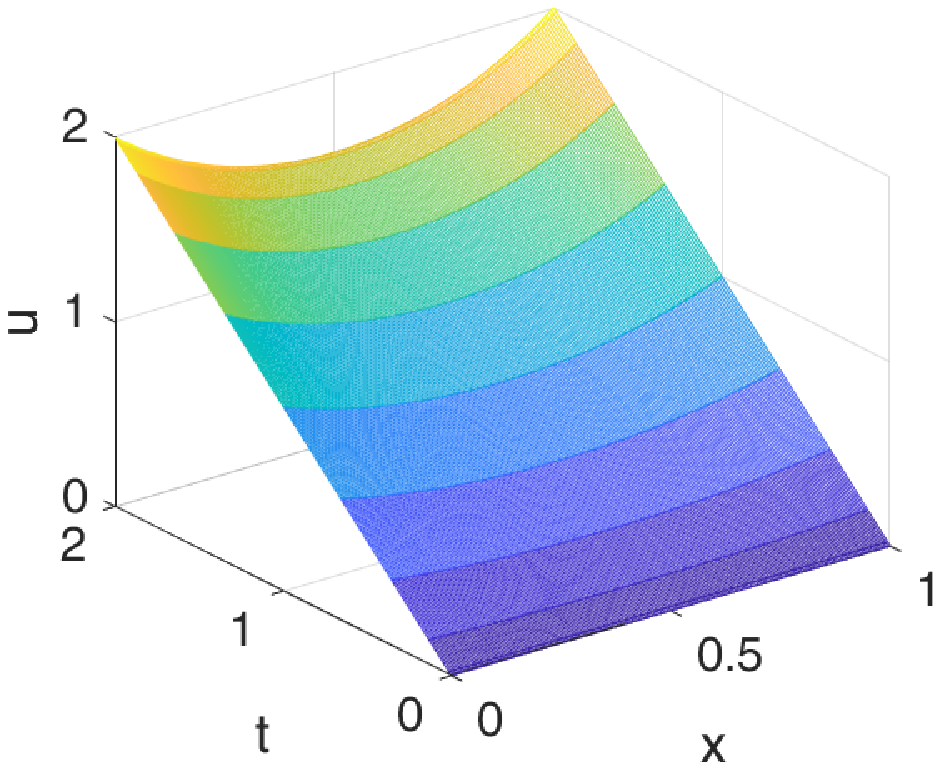}
\includegraphics[scale=0.7]{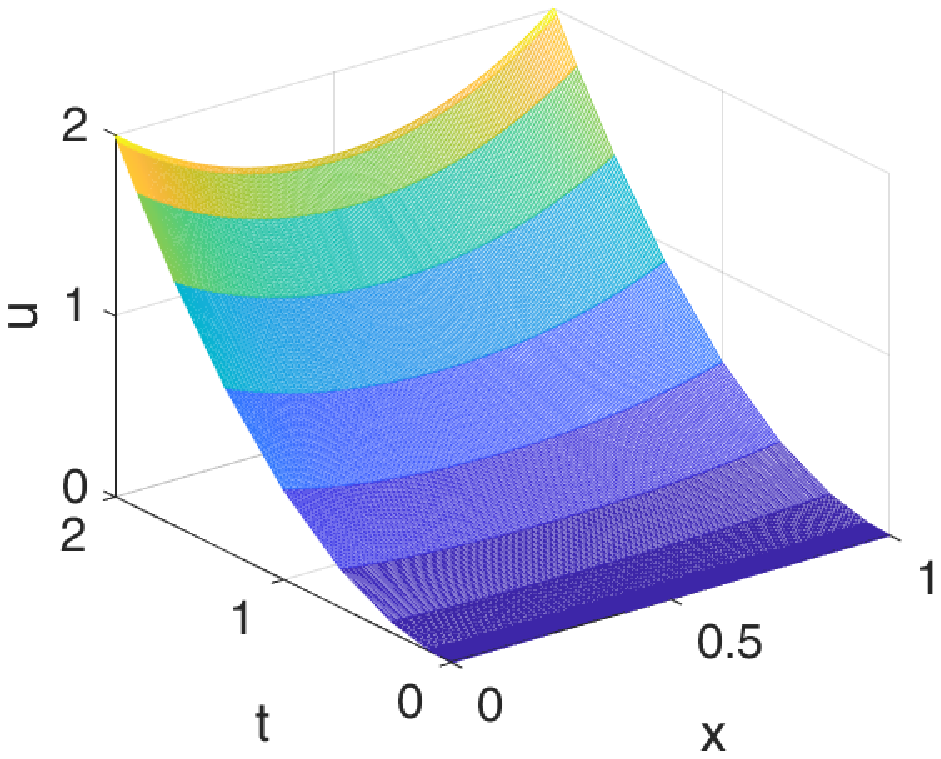}
\caption{Nonlinear problem. Numerical solution with $N=10$, $\Delta x=0.005$, $\alpha=0.5$ (left) and $\alpha=1.5$ (right).}\label{NLdiff3d}
\end{center}
\end{figure}
\begin{figure}[p]
\begin{center}
	\begin{tikzpicture}

\begin{axis}[%
width=1.257in,
height=1.073in,
at={(1.133in,2.333in)},
scale only axis,
xmin=2,
xmax=12,
xlabel style={font=\color{white!15!black}},
xlabel={$N$},
ymode=log,
ymin=1e-15,
ymax=1,
ylabel style={font=\color{white!15!black}},
ylabel={Sol Err},
yminorticks=true,
title={$\alpha=0.1$},
axis background/.style={fill=white},
xmajorgrids,
ymajorgrids,
yminorgrids,
legend style={legend cell align=left, align=left, draw=white!15!black}
]
\addplot [color=red, line width=1.5pt, mark=triangle, mark options={solid, rotate=180, red}]
  table[row sep=crcr]{%
1	0.0195203169075958\\
2	0.00145030498461877\\
3	0.000815105895982609\\
4	5.50261564219582e-05\\
5	3.01627500937508e-05\\
6	6.80250839008492e-06\\
7	4.51119893635942e-06\\
8	1.69975135953667e-06\\
9	1.05555923878995e-06\\
10	6.77945311078692e-07\\
11	2.52876630618243e-07\\
12	3.85634111625139e-07\\
};

\end{axis}
\end{tikzpicture}%
\hspace{-0.3cm}	\begin{tikzpicture}

\begin{axis}[%
width=1.257in,
height=1.073in,
at={(1.133in,2.333in)},
scale only axis,
xmin=2,
xmax=12,
xlabel style={font=\color{white!15!black}},
xlabel={$N$},
ymode=log,
ymin=1e-15,
ymax=1,
yminorticks=true,
title={$\alpha=0.5$},
axis background/.style={fill=white},
xmajorgrids,
ymajorgrids,
yminorgrids,
legend style={legend cell align=left, align=left, draw=white!15!black}
]
\addplot [color=red, line width=1.5pt, mark=triangle, mark options={solid, rotate=180, red}]
  table[row sep=crcr]{%
1	5.02429586778808e-15\\
2	3.50380476589407e-15\\
3	3.46133283549042e-15\\
4	4.01526397020849e-15\\
5	3.67549221219551e-15\\
6	3.65531552429895e-15\\
7	3.97825613944057e-15\\
8	3.91579926007877e-15\\
9	6.32339996626736e-15\\
10	4.08224039257781e-15\\
11	3.76167416018478e-15\\
12	1.02909949240062e-14\\
};

\end{axis}
\end{tikzpicture}%
\hspace{-0.3cm}	\begin{tikzpicture}

\begin{axis}[%
width=1.257in,
height=1.073in,
at={(1.133in,2.333in)},
scale only axis,
xmin=2,
xmax=12,
xlabel style={font=\color{white!15!black}},
xlabel={$N$},
ymode=log,
ymin=1e-15,
ymax=1,
yminorticks=true,
axis background/.style={fill=white},
title={$\alpha=0.9$},
xmajorgrids,
ymajorgrids,
yminorgrids,
legend style={legend cell align=left, align=left, draw=white!15!black}
]
\addplot [color=red, line width=1.5pt, mark=triangle, mark options={solid, rotate=180, red}]
  table[row sep=crcr]{%
1	0.136747297974531\\
2	0.0344967599323555\\
3	0.00889239966562753\\
4	0.00788908243594297\\
5	0.00342940458372313\\
6	0.000852925396445238\\
7	0.000115878189938765\\
8	7.24137858580057e-05\\
9	4.31223610514301e-05\\
10	4.37397635723289e-05\\
11	1.11519972612039e-05\\
12	2.24183343495833e-05\\
};

\end{axis}
\end{tikzpicture}%
\\
	\begin{tikzpicture}
\begin{axis}[%
width=1.257in,
height=1.073in,
at={(1.133in,2.333in)},
scale only axis,
xmin=2,
xmax=12,
xlabel style={font=\color{white!15!black}},
xlabel={$N$},
ymode=log,
ymin=1e-15,
ymax=1,
yminorticks=true,
ylabel style={font=\color{white!15!black}},
ylabel={Sol Err},
axis background/.style={fill=white},
title style={font=\bfseries},
title={$\alpha{=1.1}$},
xmajorgrids,
ymajorgrids,
yminorgrids,
legend style={legend cell align=left, align=left, draw=white!15!black}
]
\addplot [color=red, line width=1.5pt, mark=triangle, mark options={solid, rotate=180, red}]
  table[row sep=crcr]{%
1	0.0449695958194733\\
2	0.0128074850184744\\
3	0.000707188116505385\\
4	0.00298693395520618\\
5	0.0016267130255285\\
6	0.000279092426457771\\
7	0.000151861010525671\\
8	0.000118572262901487\\
9	3.64534421932558e-05\\
10	1.10349474938994e-05\\
11	3.28370722761403e-06\\
12	4.04237366769791e-06\\
};

\end{axis}
\end{tikzpicture}%
	\begin{tikzpicture}

\begin{axis}[%
width=1.257in,
height=1.073in,
at={(1.133in,2.333in)},
scale only axis,
xmin=2,
xmax=12,
xlabel style={font=\color{white!15!black}},
xlabel={$N$},
ymode=log,
ymin=1e-15,
ymax=1,
yminorticks=true,
ylabel style={font=\color{white!15!black}},
axis background/.style={fill=white},
title style={font=\bfseries},
title={$\alpha=1.5$},
xmajorgrids,
ymajorgrids,
yminorgrids,
legend style={legend cell align=left, align=left, draw=white!15!black}
]
\addplot [color=red, line width=1.5pt, mark=triangle, mark options={solid, rotate=180, red}]
  table[row sep=crcr]{%
1	0.286576491552045\\
2	0.126420258728108\\
3	0.0118421300325731\\
4	0.054059019761892\\
5	0.0133959403687655\\
6	0.0114302584674477\\
7	0.00302890973060446\\
8	0.00188140781388482\\
9	0.00040505007653147\\
10	0.000200060656320793\\
11	0.000277684735834782\\
12	0.000257467028891173\\
};

\end{axis}
\end{tikzpicture}%
	\begin{tikzpicture}

\begin{axis}[%
width=1.257in,
height=1.073in,
at={(1.133in,2.333in)},
scale only axis,
xmin=2,
xmax=12,
xlabel style={font=\color{white!15!black}},
xlabel={$N$},
ymode=log,
ymin=1e-15,
ymax=1,
yminorticks=true,
ylabel style={font=\color{white!15!black}},
axis background/.style={fill=white},
title style={font=\bfseries},
title={$\alpha=1.9$},
xmajorgrids,
ymajorgrids,
yminorgrids,
legend style={legend cell align=left, align=left, draw=white!15!black}
]
\addplot [color=red, line width=1.5pt, mark=triangle, mark options={solid, rotate=180, red}]
  table[row sep=crcr]{%
1	0.424094292900454\\
2	0.468796548239649\\
3	0.108729142067097\\
4	0.191123695287828\\
5	0.0475713095492379\\
6	0.01833832544794\\
7	0.0182738449969074\\
8	0.00577756353312567\\
9	0.00718733643168301\\
10	0.00430451766597249\\
11	0.00639811898163244\\
12	0.00331995397350142\\
};

\end{axis}
\end{tikzpicture}%
\end{center}
	\caption{Nonlinear problem. Rate of convergence in time. $\Delta x=0.005$ (logarithmic scale on $y$-axis).}
	\label{NL1}
\end{figure}
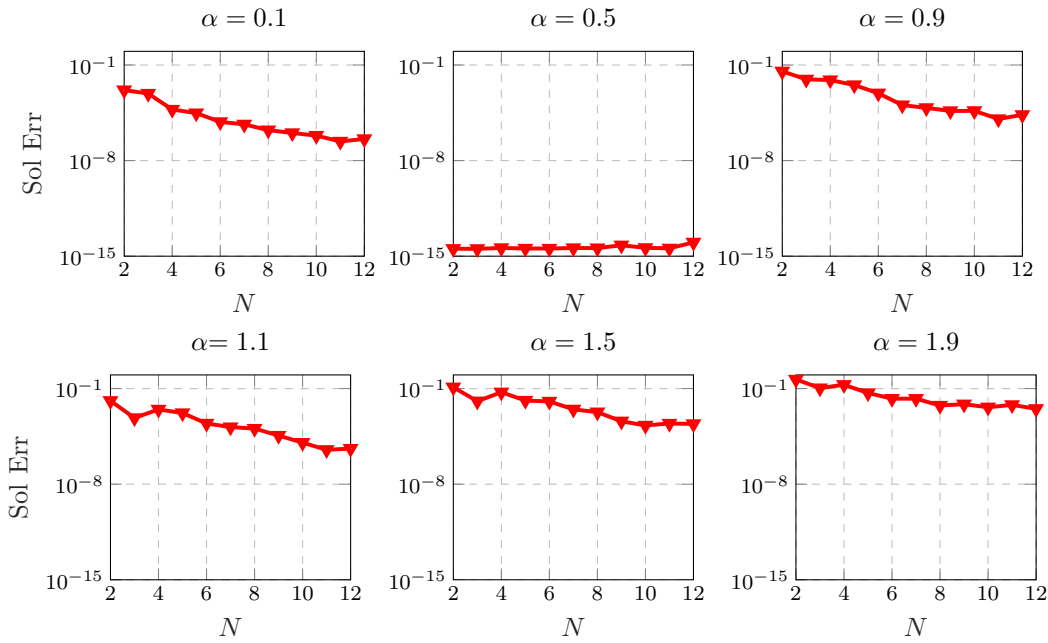

The numerical solutions obtained for $\alpha=0.5$ and $\alpha=1.5$ are shown in Figure~\ref{NLdiff3d}. The convergence of the method is analysed in Figure~\ref{NL1} where we plot the logarithm of the error in the solution against $N$ for $N=2,\ldots,12$. Also in this case the rate of convergence in time is exponential, except for $\alpha=0.5$, where the method is exact for all values of $N$.
\section{Conclusions}\label{sec:concl}
In the present paper we have investigated a class of time fractional diffusion PDEs of arbitrary fractional order. The purpose is twofold. On one hand, we have derived sufficient conditions to find conservation laws of a FDE of this kind, and derived an analogue result for a numerical method to have discrete conservation laws. On the other hand, we have generalised the spectral method in \cite{bcdp17} to approximate Caputo and Riemann-Liouville derivatives of arbitrary order. This method has been coupled with a finite difference approximation in space, and proved to have conservation laws. In the cases of subdiffusion and superdiffusion, we have derived the expressions of the conservation laws and performed numerical tests to confirm the theoretical findings and show the convergence of the method.

\subsection*{Acknowledgements}
The authors are members of the GNCS group. This work is supported by GNCS-INDAM project and by the Italian Ministry of University and Research, through the PRIN 2017 project (No. 2017JYCLSF) “Structure preserving approximation of evolutionary problems”.
\bibliographystyle{plain}
\bibliography{bibfile}

\end{document}